\documentclass[oneside,english,11pt]{amsart}
\usepackage[latin1]{inputenc}

\usepackage{amsfonts}
\usepackage{amsthm}
\usepackage{amssymb}
\usepackage{graphicx}
\usepackage[all]{xy}
\usepackage{enumitem}
\usepackage{amsmath,bm}
\usepackage{graphpap}

\newtheorem{aff}{Assertion}

\newtheorem{teo}{Theorem}[section]
\newtheorem{prop}[teo]{Proposition}
\newtheorem{ddef}[teo]{Definition}
\newtheorem{example}[teo]{Example}
\newtheorem{cor}[teo]{Corollary}

\newtheorem{lem}[teo]{Lemma}
\newtheorem*{cor*}{Corollary}

\newtheorem*{teo*}{Theorem}

\newtheorem*{assertion}{Assertion}

\newcommand{\cpt}[1]{\mathbb{C}P^{2}}

\newcommand{\pe}{\mathbb{P}}

\newcommand{\sing}{{\rm Sing}}

\newcommand{\Z}{\mathbb{Z}}

\newcommand{\C}{\mathbb{C}}

\newcommand{\F}{\mathcal{F}}
\newcommand{\G}{\mathcal{G}}

\newcommand{\cl}[1]{\mbox{$\mathcal{#1}$}}

\newcommand{\codim}{{\rm codim}}

\newcommand{\tang}{{\rm Tang}}

\newcommand{\Pf}{\mathfrak{X}_{sd}}

\newcommand{\ds}[1]{\displaystyle{#1}}

\begin{document}

\setcounter{section}{0}
\setcounter{teo}{0}
\setcounter{exe}{0}

\author{Dan\'ubia Junca \& Rog\'erio   Mol   }
\title{Holomorphic vector fields tangent to   foliations in dimension three}

\footnotetext[1]{ {\em 2010 Mathematics Subject Classification:}
 32S65, 37F75}
 \footnotetext[2]{{\em
Keywords.} Holomorphic foliations,  holomorphic vector fields, pencil of foliations}
\footnotetext[3]{First author supported by a CAPES Sandwich Doctorate scholarship. Second author supported by Universal/CNPq. }

\begin{abstract} This article studies germs of holomorphic vector fields
at $(\C^{3},0)$ that are tangent to holomorphic foliations of codimension one.
Two situations are considered. First, we assume hypotheses   on the reduction of singularities
of the vector field --- for instance, that the final models  belong to a   family of vector fields
 whose eigenvalues of the linear part   satisfy a condition of non-resonance ---
in order to conclude that the foliation is of complex hyperbolic type, that is, without saddle-nodes in its
reduction of singularities. In the second part, we prove that a vector field that
is tangent to three independent foliations is tangent to a whole pencil of foliations --- hence, to infinitely many foliations --- and,
as a consequence, it leaves invariant a germ of analytic surface. This final part is based on a local
version of a well-known characterization of pencils of foliations of codimension one in projective spaces.
\end{abstract}

\maketitle


\section{Introduction}

The main goal of this article is to study, at the origin of $\C^{3}$,
germs of holomorphic vector fields that are \emph{tangent} to holomorphic foliations of codimension
one.
 If $X$ as a germ of holomorphic vector field at $(\C^{3},0)$,
inducing a germ of singular holomorphic foliation of dimension one $\F$,
and $\omega$ is a germ of  holomorphic $1-$form which satisfies the
integrability condition --- $\omega \wedge d \omega =0$ ---
inducing a germ of singular holomorphic foliation of codimension one $\G$,
we say that $X$ (or $\F$) is tangent to $\omega$ (or to $\G$) if the
orbits of $X$ are entirely contained in the two-dimensional leaves of $\omega$,
wherever both objects are defined. We also say that $\omega$ (or  $\G$) is \emph{invariant}
by $X$ (or by $\F$). In algebraic terms, this is identified by the vanishing of the contraction of $\omega$ by $X$,
that is, $i_{X} \omega = 0$.

One interesting point  is   that, due to the    the integrability condition of   $\omega$,
  not every   germ of holomorphic vector field   at $(\C^{3},0)$
 is   tangent to a homomorphic   foliation. Examples of this situation
 are presented in  two
 recent studies where the configuration of tangency between a vector field
and a foliation is considered.
The first one, by F. Cano and C. Roche \cite{cano2014}, asserts that  a germ of holomorphic vector field $X$
at $(\C^{3},0)$ tangent to a foliation has a reduction of singularities. This means that,  after a finite sequence of
blow-ups with invariant centers (points or regular invariant  curves in general position with
the reduction divisor), the one-dimensional foliation induced by $X$ is transformed into
one  for which all singularities are \emph{elementary}, meaning
that they are locally defined by vector fields with non-nilpotent linear part.
Vector fields in a     family   proposed by F. Sanz and F. Sancho (see also \cite{cano2014})
do not admit a reduction of singularities as above and, hence, are not tangent
to any germ of holomorphic foliation.
In the second one, by D. Cerveau and A. Lins Neto \cite{cerveau2018},
it is proved that  a germ of holomorphic vector field with
isolated singularity at $0 \in \C^{3}$ that  is tangent to a holomorphic foliation always
   admits
a separatrix, that is, an invariant analytic curve. As a consequence,  vector fields in the
family of X.Gómex-Mont and I. Luengo \cite{gomez-mont1992}, which do not posses separatrices,
are not   tangent to holomorphic foliations.

The afore mentioned studies suggest that consequences of geometric nature
 arise when there is  tangency between a vector field and
a foliation.
This perception is the main motivation for this article and  shall
be developed  in two different and independent approaches.
First, in Section \ref{section-sp} we present    the concept of strongly diagonalizable
germ of vetor field  (Def. \ref{def-sp}), meaning that   its linear part   has    eigenvalues
that     do not satisfy any non-trivial relation of linear dependency with integer coefficients.
We characterize the $1-$forms that are tangent  to vector fields of this type  (Prop. \ref{proposicaoch}). In the main result of the section, we consider a germ of integrable holomorphic $1-$form $\omega$ at $(\C^{3},0)$ that leaves invariant a germ of holomorphic vector field $X$,  putting the following hypotheses on the reduction of singularities of $X$ (which exists by Cano-Roche's result):
it is composed only by punctual blow-ups (we say in this case $X$ has an \emph{absolutely isolated} singularity
$0 \in \C^{3}$), the   divisor associated to this sequence of blow-ups is invariant
by the transformed foliation (that is, $X$ is \emph{non-dicritical}) and that
all final models are strongly diagonalizable singularities. Under these assumptions on
$X$, we prove  in Theorem \ref{teo-sp-ch}  that $\omega$ defines a foliation of codimension one  which
is   \emph{complex hyperbolic} --- this notion is an extension for foliations of codimension one in higher dimensions
 of the widely studied concept  of \emph{generalized curve} foliations in dimension two
 (see  definitions in Section \ref{section-preliminaries}).

In Section \ref{section-pencil}, we investigate the   situation where a germ of holomorphic
 vector field at $(\C^3,0)$ is tangent to three independent foliations, induced by
 germs of integrable holomorphic
 $1-$forms $\omega_{1}, \omega_{2}$ and $\omega_{3}$. We prove that, in this case,
 up to multiplication by germs of functions in $\cl{O}_{3}$,
 these $1-$forms define a \emph{pencil of integrable $1-$forms}, that is, a two-dimensional
 linear space --- which becomes one-dimensional when projectivized --- in the space of integrable holomorphic $1-$forms. As a consequence, $X$ is tangent to the infinitely many
 integrable $1-$forms in this pencil.
In Theorem \ref{pencil-foliations}, we present a geometric characterization of pencil of integrable   $1-$forms,
stated in the more general context of foliations at $(\C^{n},0)$, $n \geq 3$,
which is a local version of the one given by D. Cerveau in \cite{cerveau2002}.
It asserts that there is  a closed meromorphic $1-$form $\theta$ such that $d \omega = \theta \wedge \omega$ for all $1-$forms $\omega$ in the pencil  or the \emph{axis foliation}--- that is, the unique  foliation of codimension two
 that is tangent to all $1-$forms in the   pencil --- has a meromorphic first integral.
 As a consequence, the axis foliation always admits an invariant hypersurface.
 This, translated into our original three-dimensional context, gives us that a  germ of vector field at $(\C^{3},0)$ that is tangent to
 three independent germs of holomorphic foliations leaves invariant a germ of analytic surface.

 This paper contains partial results of the Ph.D thesis of the first author.
She is thanks N. Corral and the University of Cantabria for hospitality during
the development of part of this research.

\section{Preliminaries}

\label{section-preliminaries}

In local analytic coordinates $(x_{1}, x_{2}, x_{3})$
at $(\C^{3},0)$, we denote a germ of holomorphic  vector field at $(\C^{3},0)$
by
$$X = A \frac{ \partial}{\partial x_{1}} + B \frac{ \partial}{\partial x_{2}}
+ C \frac{ \partial}{\partial x_{3}},$$
where $A,B,C\in \cl{O}_{3}$.
We also consider the germ of singular one-dimensional foliation $\F$ whose leaves are the orbits of $X$. Thus, in order to avoid superfluous singularities, we   suppose that
$A,B$ and $C$ are without common factors. The singular sets (of $X$ or $\F$) are denoted  by
$\sing(X) = \sing(\F) = \{A= B = C = 0\}$,  being an analytic set of dimension at most one.
We  denote a germ of of holomorphic  $1-$form at $(\C^{3},0)$ by
$$\omega = a d x_{1} + b x_{2} + c dx_{3},$$
 where $a,b,c \in \cl{O}_{3}$, which
are also supposed to be without common factors. If $\omega$ is integrable in the sense of
Frobenius, that is, $\omega \wedge d \omega = 0$, it induces a germ of singular holomorphic foliation
of codimension one, denoted by $\G$. We have
$\sing(\omega) = \sing(\G) = \{a =b =c = 0\}$, also an analytic set of dimension at most one.
A \emph{separatrix} for a local holomorphic foliation  is an irreducible germ of invariant analytic variety
of the same dimension of the foliation. For   germs of holomorphic vector fields  and integrable $1-$forms   at $(\C^{3},0)$,   separatrices are  germs of
  invariant curves and surfaces, respectively.

We say that $X$ (or $\F$) is tangent to $\omega$ (or to $\G$) if
\[ i_{X} \omega = a A + b B + c C = 0 ,\]
which is equivalent to saying that, outside $\sing(X) \cup \sing(\omega)$, the orbits of $X$ are contained
in the two-dimensional leaves of  $\G$.
If the germ of vector field $X$ is tangent to the integrable $1-$form $\omega$,
then $\sing(\omega)$ is invariant by $X$
(see, for instance, \cite[Th. 1]{mol2011}). Thus, the one-dimensional components of $\sing(\omega)$ ---
if they exist --- are separatrices of $X$.

One particular example of this configuaration of tangency is provided by  vector fields with
first integrals. A non-constant germ of meromorphic --- or holomorphic --- function $\Phi$ at $(\C^{3},0)$ is a \emph{first integral}
for $X$ if, in a small neighborhood of $0 \in \C^{3}$, $\Phi$ is constant along the orbits of $X$.
This is equivalent to saying that $X$ is tangent to the foliation defined by $\Phi$, which
is induced by the holomorphic $1-$form obtained by cancelling the components of zeros and poles of the meromorphic
 $1-$form $d \Phi$ (we also say that $\Phi$ is a first integral for this foliation).
Note that in this case, any fiber of $\Phi$ accumulating to the origin is an $X$-invariant surface.
We should mention that, in the article \cite{pinheiro2014}, the authors study  germs of vector fields   at $(\C^{3},0)$ that are \emph{completely integrable} --- that is, that have two independent holomorphic first integrals --- and prove that this property is
not a topological invariant.

In dimension two, a germ of holomorphic foliation is induced in local analytic coordinates
 $(x_{1},x_{2})$ at   $(\C^{2},0)$  by
 a germ of holomorphic vector field $X = A  \partial / \partial x_{1}  + B \partial / \partial x_{2}$
 with isolated singularity at the origin (or
 by the dual $1-$form $\omega = B  dx_{1}  - Ad x_{2}$).
Recall that such a foliation has a reduction of singularities, that is,   a finite sequence of blow-ups
transforms it into one having a finite number of singularities which are  \emph{simple} or \emph{reduced} \cite{seidenberg1968}. Such a singularity is locally
 induced by a vector field  whose linear part is non-nilpotent, having two eigenvalues $\lambda_{1}$ and $\lambda_{2}$  such that, if  both   non-zero, do not satisfy any non-trivial relation of the kind
$m_{1}\lambda_{1} + m_{2} \lambda_{2} = 0$ with $m_{1},m_{2} \in \Z_{\geq 0}$. These simple singularities are called
\emph{non-degenerate}, whereas the ones having one zero eigenvalue are called \emph{saddle-nodes}.
A foliation  at $(\C^{2},0)$ is said to be of \emph{generalized curve}  type \cite{camacho1984} if it has no saddle-nodes in some (and hence in any) reduction of singularities.
Several geometric properties of a foliation of generalized curve type can be read in  its separatrices. For example,   a sequence of blow-ups that desingularizes
its set of separatrices    is also a reduction of
singularities for the foliation  itself \cite{camacho1984}.

A foliation of codimension one at $(\C^{3},0)$ also admits a reduction of singularities \cite{cerveau1992,cano2004}. This means that after a finite sequence of blow-ups with invariant centers
--- points and regular curves satisfying a condition of  general position with respect to the
 divisor ---, the foliation is transformed into one whose singularities are all   \emph{simple} or \emph{reduced}, which, in analogy with the two-dimensional case, are essentially of two kinds:
simple complex hyperbolic singularities and simple saddle-node singularities (see, for instance, the
description in \cite{cuzzuol2018}).
Recall that the \emph{dimensional type} of a codimension one foliation,    $\tau \geq 1$, is the
smallest number of variables needed to express its defining equation
in some system of analytic coordinates.
We say that a simple singularities is of  complex hyperbolic type
\cite{cano2015} if  there
are analytic coordinates $(x_{1},x_{2},x_{3})$ at $0 \in \C^{3}$ in which the foliation is defined by
a holomorphic $1-$form $\omega$ whose   terms of lowest order are: \smallskip
\begin{itemize}
\item \ $\ds{x_{1} x_{2} \left( \lambda_{1} \frac{d x_{1}}{x_{1}}+  \lambda_{2} \frac{d x_{2}}{x_{2}}
 \right)}$, if $\tau = 2$, \medskip
\item \ $\ds{x_{1} x_{2} x_{3} \left(  \lambda_{1}   \frac{d x_{1}}{x_{1}}+  \lambda_{2}   \frac{d x_{2}}{x_{2}}
+ \lambda_{3} \frac{d x_{3}}{x_{3}}  \right)}$, if $\tau = 3$,
\end{itemize}
where the residues $\lambda_{i} \in \C^{*}$   are non-resonant,  that is
there are no non-trivial relations of the kind    $m_{1} \lambda_{1}+ m_{2} \lambda_{2}= 0$ (for $\tau = 2$)
or $m_{1} \lambda_{1}+ m_{2} \lambda_{2}+ m_{3}\lambda_{3}= 0$ (for $\tau = 3$),
with $m_{i} \in \Z_{\geq 0}$.  Simple complex hyperbolic singularities, when $\tau =2$, correspond
to simple non-degenerate singularities of foliations in dimension two.

Now,  a foliation $\G$ at $(\C^{3},0)$ induced by an integrable holomorphic $1-$form $\omega$
is of \emph{complex hyperbolic} type if it satisfies one of the two equivalent properties \cite{cano2015}:
\begin{itemize}
\item There exists a \emph{complex hyperbolic} reduction of singularities for $\G$, that is, one  for which all final models are simple complex hyperbolic.
In this case, every reduction of singularities of $\G$ will be complex hyperbolic.
\item for every holomorphic map $\phi: (\C^{2},0) \to(\C^{3},0)$ generically transversal to $\G$ (that is,
such that $\phi^{*} \omega$ has an isolated singularity at $0 \in \C^{2}$) the foliation $\pi^{*} \F$ induced
by $\phi^{*} \omega$ is of generalized curve type.
\end{itemize}
Complex hyperbolic foliations are the    three-dimensional counterparts of generalized curve foliations. It is thus expected that some geometric properties enjoyed by the latter also
have a formulation for the former. For instance, in the non-dicritical case (that is, if the reduction divisor is invariant by the transformed foliation), its proved in \cite{fernandez2009} that
a complex hyperbolic foliation becomes reduced once its set of separatrices is desingularized.

\section{Strongly diagonalizable vector fields}
\label{section-sp}
In this section we introduce the notion of strongly diagonalizable germs of vector fields.  
A vector field  in this family has a linear part with eigenvalues   satisfying  a  hypothesis of non-resonance, being,
 as a consequence,  linearizable in formal coordinates.
We show that by assuming hypothesis   on
the reduction of singularities of a germ of  vector field at $(\C^{3},0)$
  that is tangent to
a   holomorphic foliation of codimension one --- among them, that the final models belong to
this family of strongly diagonalizable vector fields --- we are able to
conclude that the foliation is of complex hyperbolic type.

Recall that  a vector $\alpha = (\alpha_1, \cdots, \alpha_n)
\in  \mathbb{C}^n $, $n \geq 2$,  is \emph{non-resonant}
if there are no relations of the form $\alpha_{j} = \ell_1 \alpha_1 + \cdots +\ell_{n} \alpha_n = 0$,
with $\ell_1, \ldots, \ell_n \in \Z_{\geq 0}$ satisfying $\sum_{j=1}^{n}\ell_i \geq 2$.
A classical result asserts that a  germ of complex analytic vector field $X$ at $(\C^{n},0)$ whose associated
eigenvalues  (i.e. those of its linear part $DX(0)$) are non-resonant is linearizable in formal coordinates.
We remark that these linearizing coordinates can be taken to be analytic if these eigenvalues belong to the Poincaré domain --- i.e. the set of vectors
$\alpha = (\alpha_1, \cdots, \alpha_n)
\in  \mathbb{C}^n $ such that the origin $0 \in \C^{n}$ is not in
the convex hull of $\{\alpha_1, \cdots, \alpha_n\}$.

In the next definition, we   work with the following  notion:
a vector $ \alpha = (\alpha_1, \cdots, \alpha_n)
\in  \mathbb{C}^n $ is said to be \emph{strongly non-resonant}
if there are no non-trivial relations of the form $\ell_1 \alpha_1 + \cdots +\ell_{n} \alpha_n = 0$,
with $\ell_1, \ldots, \ell_n \in \Z$. Such a non-trivial relation will be called      \emph{strong resonance}. We have:

\begin{ddef}
\label{def-sp}
{\rm
A germ of holomorphic vector field  at $(\mathbb{C}^n,0)$, $n \geq 2$,  is said to be  \emph{strongly diagonalizable} if
 its   associated eigenvalues are
strongly non-resonant.
}\end{ddef}
We denote the family of strongly diagonalizable vector fields by $\Pf$. 
Clearly, a vector field in $\Pf$ satisfies the usual condition of non-resonance, being
linearizable in formal coordinates. Further, the
  associated  eigenvalues  are nonzero and pairwise distinct,
implying   that its linear part   is  diagonalizable.
We  also say that a germ of one-dimensional foliation $\F$ at $(\mathbb{C}^n,0)$ is strongly diagonalizable if it is induced by a vector field in $\Pf$. Evidently, this definition does not depend on the choice of the vector field in $\Pf$ inducing $\F$.

In the sequel, we restrain ourselves to  ambient dimension $n=3$.  Thus, if $X$ is a germ of holomorphic vector field at $(\C^{3},0)$  in
  $\Pf$, we can take local formal coordinates $(x_{1},x_{2},x_{3})$ such that
\begin{equation}
\label{eq-inv}
X=\alpha_{1}x_1\dfrac{\partial}{\partial x_1}+\alpha_{2}x_2\dfrac{\partial}{\partial x_2}+\alpha_{3}x_3\dfrac{\partial}{\partial x_3},
\end{equation}
where   $\alpha_{1}, \alpha_{2}, \alpha_{3} \in \C^{*}$ are pairwise distinct.
Note that if we choose numbers $b_{1}, b_{2}, b_{3} \in \C$, not
all of them zero, satisfying $\alpha_{1}b_{1} + \alpha_{2}b_{2} +\alpha_{3}b_{3}=0$,
then $X$ is tangent to   the formal   meromorphic $1-$form
$$\omega = x_{1}x_{2}x_{3} \left(b_{1} \frac{dx_{1}}{x_{1}}  + b_{2}\frac{dx_{2}}{x_{2}} + b_{3} \frac{dx_{3}}{x_{3}}\right).$$

We start by  proving a simple lemma:
\begin{lem}
\label{lemainvariante}
A vector field in  $X \in \Pf$ has exactly three formal smooth separatrices, which correspond to the coordinate axes in its diagonalized   form.
\end{lem}

\begin{proof} We take $X$   in its diagonal form \eqref{eq-inv}.
Suppose, without loss of generality, that an $X$-invariant curve $\gamma$ is parametrized as $\gamma(t)=(at+f(t),g(t),h(t))$  where $a\in \C^{*}$ and $f,g,h \in \hat{\mathcal{O}}_1$ are non-units, with $f$ of order at least two. We have to prove that $g=h=0$.
Suppose, for instance, $g \neq 0$. The condition of invariance is
expressed as
\[\Phi(t)(a+f^{\prime}(t),g^{\prime}(t),h^{\prime}(t))=(\alpha_{1}(at+f(t)),\alpha_{2}g(t),\alpha_{3}h(t)),\]
for some $\Phi \in \hat{\cl{O}}_{1}$ of the form $\Phi(t)=\alpha_{1} t+\rho(t)$, where   $\nu_{0}(\rho) \geq 2$.
The above equation gives us
\[
\dfrac{\alpha_{2}}{\Phi(t)}=\dfrac{g^{\prime}(t)}{g(t)}.
\]
Comparing residues in this formula, we find
$\alpha_{2}/\alpha_{1}=m$, where   $m = \nu_0(g)$. This, however, gives a strong resonance for the vector $(\alpha_{1}, \alpha_{2}, \alpha_{3})$, in contradiction with our hypothesis. Therefore $g = 0$ and, in a similar way, $h=0$, giving that  $\gamma$ is contained in the $x_1$-axis.
\end{proof}

Blow-ups preserve the family of strongly diagonalizable vector fields.   More precisely, we have the following:
\begin{lem}
Let   $ X \in \Pf$.  Then the strict transform of the one-dimensional foliation induced by  $X$ by  a blow-up  with smooth invariant center   is locally  given by vector fields in $\Pf$.
\end{lem}

\begin{proof}
Take formal diagonalizing coordinates    for $X$.
For a punctual blow-up at $0\in \C^{3}$, consider
  coordinates $x_1^{\ast}= x_1,x_2^{\ast}=x_2/ x_1$ and $x_3^{\ast}=x_3/ x_1$. In these coordinates, the   strict  transform of $ X $ is
\[
\widetilde{X}=\alpha_{1}x_1^{\ast}\dfrac{\partial}{\partial x_1^{\ast}}+(\alpha_{2}-\alpha_{1})x_2^{\ast}\dfrac{\partial}{\partial x_2^{\ast}}+(\alpha_{3}-\alpha_{1})x_3^{\ast}\dfrac{\partial}{\partial x_3^{\ast}},
\]
having an isolated singularity at $(x_1^{\ast}, x_2^{\ast}, x_3^{\ast}) = (0,0,0)$.  We only have to check that the eigenvalues of $\widetilde{X}$ are strongly non-resonant.
However, a relation of the kind
\[ 0 = c_{1}\alpha_{1} +  c_{2}(\alpha_{2}-\alpha_{1})+ c_{3}(\alpha_{3}-\alpha_{1}) =
(c_{1}-c_{2}-c_{3})\alpha_{1}+c_{2}\alpha_{2}+c_{3}\alpha_{3},
\]
for $c_{1}, c_{2}, c_{3} \in \Z$, is possible if and only if $c_{1}=c_{2}=c_{3}=0$, since the eigenvalues
of $X$ are strongly non-resonant.

In the case of a monoidal blow-up, its     smooth invariant center must be one of the coordinate axes, by Lemma \ref{lemainvariante}.
For instance, fixing the $ x_3 $-axis as the blow-up center  and taking blow-up charts $ x_1 = x_1^{\ast}, x_2 = x_2^{\ast}$ and $ x_3^{\ast} = x_3 / x_2 $,     the strict transform of  $X$ is
\[
\tilde{X}=\alpha_{1} x_1^{\ast}\dfrac{\partial}{\partial x_1^{\ast}}+\alpha_{2} x_2^{\ast}\dfrac{\partial}{\partial x_2^{\ast}}+(\alpha_{3}-\alpha_{2})x_3^{\ast}\dfrac{\partial}{\partial x_3^{\ast}}.
\]
Again, the absence of strong resonances for the eigenvalues of $\tilde{X}$ follows from that of $X$.
\end{proof}

As a consequence, we have:
\begin{cor}
The only formal separatrices of a vector field in $\Pf$ are those corresponding to   the   coordinate axes
in its diagonal form.
\end{cor}

\begin{proof}
Let $\gamma$ be a formal separatrix  for $X \in \Pf$. If $\gamma$ is smooth, this has been proved in  Lemma \ref{lemainvariante}. If $\gamma$ is singular, we desingularize it through a sequence of punctual blow-ups. By the previous lemma,   the strict transform of the foliation induced by $X$ has local models in $\Pf$. Its smooth invariant curves are either in the desingularization divisor or are contained in the strict transforms of the coordinate axis in diagonalizing coordinates for $X$.  The transform of $ \gamma $ is evidently not  in the desingularization divisor.  This means that
$\gamma$ is smooth contained in one of the coordinate axes, which is not our case.
\end{proof}

Our objective now is to prove the following result:
\begin{prop}
\label{proposicaoch}
Let $\omega $ be a germ of integrable holomorphic $1-$form at $ (\mathbb{C}^ 3,0)$ with
$\codim\, \sing(\omega) \geq 2$. Suppose that $\omega$ is invariant by a  vector field in $X \in \Pf$. Then, in formal diagonalizing coordinates for $X$ and up to multiplication by a unit in $ \hat{\mathcal{O}}_3 $, we have either
\[
\omega=x_1x_2\left(b_1\dfrac{dx_1}{x_1}+b_2\dfrac{dx_2}{x_2}\right) \tag{I}
\]
or
\[\omega=x_1x_2x_3\left( b_1\dfrac{dx_1}{x_1}+b_2\dfrac{dx_2}{x_2}+b_3\dfrac{dx_3}{x_3}\right),\tag{II}\]
where $b_1,b_2,b_3\in\mathbb{C}^{\ast}$.
\end{prop}
\begin{proof}
Fix $(x_{1},x_{2},x_{3})$ formal diagonalizing coordinates for $X$ as in \eqref{eq-inv} and write
\begin{equation}
\label{formaomega}\omega=adx_1+bdx_2+cdx_3,
\end{equation}
where $a,b,c\in \hat{\mathcal{O}}_3$ are without common factors.   Since $X$ is tangent to $\omega$, the contraction of $\omega$ by $X$ gives
\begin{equation}
0=i_X\omega=\alpha_{1}x_1a+\alpha_{2}x_2b+\alpha_{3}x_3c.
\label{continhas}
\end{equation}
The integrability condition in its turn reads
\begin{equation}
0 = \omega\wedge d\omega = a(c_{x_2}-b_{x_3})+b(-c_{x_1}+a_{x_3})+c(b_{x_1}-a_{x_2}).
\label{integrabilidade}
\end{equation}
The differentiation of  (\ref{continhas}) with respect to each of the variables $ x_1, x_2 $ and $ x_3 $  produces the following set of equations:
\begin{eqnarray}
\alpha_{1}a+\alpha_{1}x_1a_{x_1}+\alpha_{2}x_2b_{x_1}+\alpha_{3}x_3c_{x_1}=0;    \\
\label{c1}
\alpha_{1}x_1a_{x_2}+\alpha_{2}b+\alpha_{2}x_2b_{x_2}+\alpha_{3}x_3c_{x_2}=0 ; \medskip \\
\label{c2}
\alpha_{1}x_1a_{x_3}+\alpha_{2}x_2b_{x_3}+\alpha_{3}x_3c_{x_3}+\alpha_{3}c=0. \medskip
\label{c3}
\end{eqnarray}
We have the following (this was shown to us by M. Fernández-Duque):
\begin{aff}
In the above conditions, $X$ leaves invariant each ratio of coefficients of $\omega$.
\end{aff}
\noindent \emph{Proof of the Assertion.}
In fact,
\begin{eqnarray}
b^2X\left(a/b\right)&=& bX(a)-aX(b)\nonumber\\
&=& b(\alpha_{1} x_1a_{x_1}+\alpha_{2} x_{2}a_{x_2}+\alpha_{3}x_3 a_{x_3})-a(\alpha_{1} x_1b_{x_1}+\alpha_{2} x_2b_{x_2}+\alpha_{3} x_3 b_{x_3})\nonumber  \\
&=& b(-\alpha_{1} a-\alpha_{2} x_2b_{x_1}-\alpha_{3} x_3c_{x_1})+b\alpha_{2} x_2a_{x_2}+b\alpha_{3} x_3a_{x_3}-a\alpha_{1} x_1b_{x_1}\nonumber\\ &&\ \ \ \ \   -a\alpha_{2} x_2b_{x_2}-a\alpha_{3} x_3b_{x_3} \ \ \   (\text{by \eqref{c1}} ) \nonumber\\
&=& b\alpha_{3} x_3(a_{x_3}-c_{x_1})-ab\alpha_{1} -b\alpha_{2} x_2b_{x_1}+b\alpha_{2}x_2a_{x_2} -a\alpha_{1} x_1b_{x_1}\nonumber\\ && \ \ \ \ \   -a\alpha_{2} x_2b_{x_2}-a\alpha_{3} x_3b_{x_3}\nonumber\\
&=& b\alpha_{3} x_3(a_{x_3}-c_{x_1})+a\alpha_{3}x_3(c_{x_2}-b_{x_3})+ab(\alpha_{2}-\alpha_{1}) \nonumber\\ &&\ \ \ \ \ + b\alpha_{2}x_2(a_{x_2}-b_{x_1})+a\alpha_{1} x_1(a_{x_2}-b_{x_1})\nonumber \ \ \ (\text{by \eqref{c2}}) \\
&=& \alpha_{3}x_3(c(a_{x_2}-b_{x_1}))+ab(\alpha_{2}-\alpha_{1})+(a_{x_2}-b_{x_1})(a\alpha_{1} x_1+b\alpha_{2} x_2) \ \ (\text{by  \eqref{integrabilidade}}) \nonumber\\
&=& ab(\alpha_{2}-\alpha_{1})+(a_{x_2}-b_{x_1})(\alpha_{1} ax_1+\alpha_{2} bx_2+\alpha_{3} cx_3)\ \ (\text{by \eqref{continhas}}) \nonumber\\ &=& ab(\alpha_{2}-\alpha_{1}). \nonumber
\end{eqnarray}
That is, $X\left(a/b\right)=(\alpha_{2}-\alpha_{1})a/b$. In a similar way, we find $X\left(a/c\right)=(\alpha_{3}-\alpha_{1})a/c$ and $X\left(b/c\right)=(\alpha_{3}-\alpha_{2})b/c$,
proving the assertion. \qed

\medskip
We have just found that
\begin{equation*}
X(a/b)=\mu_1a/b,  \qquad X(a/c)= \mu_2a/c \qquad \text{and} \qquad X(b/c)= \mu_3b/c,
\end{equation*}
where $\mu_{1} = \alpha_{2}-\alpha_{1}$, $\mu_{2} = \alpha_{3}-\alpha_{1}$ and $\mu_{3} = \alpha_{3}-\alpha_{2}$ are in $\C^{*}$.
These equations can be rewritten as
\begin{equation}
bX(a)-aX(b)=\mu_1ab ,  \ \ \ cX(a)-aX(c)=\mu_2ac \ \ \ \text{and} \ \ \
cX(b)-bX(c)=\mu_3bc.
\label{ultimahora}
\end{equation}
The first of these equations is equivalent to $bX(a)=a(X(b)+\mu_1b)$, where we can see that the factors of $a$ that do not divide $b$ do divide $X(a)$.
Similarly, from second equation  we have   $cX(a)=a(X(c)+\mu_2c)$, allowing us to conclude that the factors of $a$ that do not divide  $c$ do divide $X(a)$. Since $a$, $b$ and $c$ do not have common factors we conclude that $a$ divides $X(a)$. Analogously, $b$ divides $X(b)$ and $c$ divides $X(c)$.
Therefore, we can find functions $R_1,R_2,R_3\in \hat{\mathcal{O}}_3$ such that
\[   X(a)=R_1a, \qquad
X(b)=R_2b \qquad  \text{and} \qquad
X(c)=R_3c.\]
 From equations \eqref{ultimahora}  we have
\[
R_1-R_2=\mu_1, \qquad
R_1-R_3=\mu_2 \qquad \text{and} \qquad
R_2-R_3=\mu_3.
\]
For $i=1,2,3$, we write $ R_i=(\lambda_i+f_i)$, where $\lambda_i\in\mathbb{C}$ and $f_i\in \hat{\mathcal{O}}_3$ is  a non-unity. From the above equations,
we have
\[ \lambda_1-\lambda_2=\mu_1=\alpha_{2}-\alpha_{1},  \qquad
\lambda_1-\lambda_3=\mu_2=\alpha_{3}-\alpha_{1}, \qquad
\lambda_2-\lambda_3=\mu_3=\alpha_{3}-\alpha_{2}
\]
and $f_1=f_2=f_3=f$.

Suppose that $a\neq 0$ and denote by $a_{\nu}$ be its initial part, that is, the homogeneous part of order $\nu = \nu_{0}(a)$ of its Taylor series. Taking   initial parts in both sides of
$X(a)=(\lambda_1+f)a$
and considering the fact that the derivation by $X$ preserves the  degree ---    actually,  the multidegree --- of each monomial, we have \[X(a_{\nu})=\lambda_1a_{\nu}.\]
Further, if $\kappa = b_1x_1^ix_2^jx_3^k$ is a non-zero monomial in  $a_{\nu}$, where $b_{1} \in \C^{*}$, then
$X(\kappa)=\lambda_1 \kappa$, which gives
$\lambda_1=i\alpha_{1}+j\alpha_{2}+k\alpha_{3}\neq 0$.
Note that the fact that $(\alpha_{1},\alpha_{2},\alpha_{3})$ is strongly non-resonant implies that
$a_{\nu} = \kappa$.
\begin{aff}
\label{afirnova}
$a_{\nu}$ divides $a$.
\end{aff}
\noindent \emph{Proof of the Assertion.}
Write the    power series $a=\sum_{\ell\geq \nu}a_\ell$, where $a_\ell$ assembles the homogeneous terms of degree  $\ell$. We will show by induction that $a_{\nu} = b_1x_1^ix_2^jx_3^k$ divides each $a_\ell$. There is nothing to prove for $\ell= \nu$.  Let $m>\nu$ and suppose that  $a_{\ell}$ is divisible by $a_{\nu}$ for all $ \ell=\nu,\ldots, m-1$. Let $\varrho$ be  a monomial of $a_m$. We have $X(\varrho)=\lambda \varrho$ for some $\lambda \in \C$.  Considering the calculation in the above paragraph, since $\lambda_1$ has  already been determined by the multidegree of $a_{\nu}$, we must have  $\lambda\neq \lambda_1$.
On the other hand,
separating all monomials  of the same multidegree of $\varrho$ in the expression $ X(a) = \lambda_1a + fa $, we have
\begin{equation}
X(\varrho)=\lambda_1 \varrho+ \tilde{\varrho},
\label{xF}
\end{equation}
where $\tilde{\varrho}$ assembles all monomials coming from $fa$. Notice that $\tilde{\varrho}$ can be seen a combination of monomials of $a$ of order smaller than $m$ having monomials of $f$  as coefficients.
Hence, by the induction hypothesis, $\tilde{\varrho}$ is divisible by $a_{\nu}$.
Rewriting \eqref{xF} as $\lambda \varrho=\lambda_1 \varrho+\tilde{\varrho}$, we find $(\lambda-\lambda_1)\varrho=\tilde{\varrho}$, from where we deduce that  $\varrho$ is also
divisible by $a_{\nu}$. We then conclude that $a_{\nu}$ divides  $a_m$, proving the general step of the induction.
\qed

\medskip

Suppose that the coefficients $ a, b $, and $ c $ of $\omega$ are non-zero. By Assertion \ref{afirnova}, we can  write
\[
a=x_1^{i_1}x_2^{j_1}x_3^{k_1}(b_1+g_1), \ \  \
b=x_1^{i_2}x_2^{j_2}x_3^{k_2}(b_2+g_2) \ \ \ \text{and} \ \ \
c=x_1^{i_3}x_2^{j3}x_3^{k_3}(b_3+g_3),
\]
where $b_1,b_2,b_3\in \mathbb{C}^{\ast}$ and $g_1,g_2,g_3\in \hat{\mathcal{O}}_3$ are non-units. Since $X$ is tangent to $\omega$, we have
\[
\alpha_{1}b_1x_1^{i_1+1}x_2^{j_1}x_3^{k_1}+\alpha_{2}b_2x_1^{i_2}x_2^{j_2+1}x_3^{k_2}+
\alpha_{3}b_3x_1^{i_3}x_2^{j_3}x_3^{k_3+1}=0,
\]
which implies that
\[
i_1+1=i_2=i_3, \qquad
j_1=j_2+1=j_3 \qquad \text{and} \qquad
k_1=k_2=k_3+1.
\]
Since $a,b,c$ do not have common factors,  we find straight  that $i_1=j_2=k_3=0$, giving
\[
i_2=i_3=1, \qquad
j_1=j_3=1 \qquad \text{and} \qquad
k_1=k_2=1.
\]
Now we can write
\begin{eqnarray}
\omega &=& x_2x_3(b_1+g_1)dx_1+x_1x_3(b_2+g_2)dx_2+x_1x_2(b_3+g_3)dx_3\nonumber\\\medskip
&=&x_1x_2x_3\left((b_1+g_1)\dfrac{dx_1}{x_1}+(b_2+g_2)\dfrac{dx_2}{x_2}+(b_3+g_3)\dfrac{dx_3}{x_3}\right).
\label{novissima4}
\end{eqnarray}
Dividing equation  \eqref{novissima4}  by $1+g_1/b_1$, we can rewrite, abusing, notation, \begin{equation}
\omega=b_1x_2x_3dx_1+x_1x_3(b_2+g_2)dx_2+x_1x_2(b_3+g_3)dx_3.
\end{equation}
Let us apply the relation $X(b/a)=-\mu_1b/a$ of \eqref{ultimahora}  to this writing of $\omega$. Write
\[
\dfrac{b}{a} = \dfrac{x_1}{b_1x_2}(b_2+g_2)  =  \sum_{ i,k\geq 0,\, j\geq -1} \alpha_{ijk}x_1^ix_2^jx_3^k
\]
as a sum of meromorphic monomials.
Since the derivation by $X$ preserves monomials, we have that
 $\alpha_1 i  + \alpha_2 j + \alpha_3 k = - \mu_{1}$ whenever $\alpha_{ijk} \neq 0$. Again, the fact that $(\alpha_1 ,\alpha_2 ,\alpha_3)$ is free from strong resonances implies that $b/a$ is a monomial.    Thus $g_2=0$ and
\[
\dfrac{b}{a}=\dfrac{b_2x_1}{b_1x_2},
\]
implying $b=b_2x_1x_3$.
In an analogous way, we can also prove that $c=b_3x_1x_2$. This leads to the form (II) in the
statement of the proposition. 
The case where one of the  coefficients of $\omega$ is zero, for example, $c=0$, is treated following
 the
 same steps above, giving form (I) in the
statement.
\end{proof}

\begin{prop}
\label{propprausarnolema}
Let $\omega$ be a germ of integrable holomorphic $1-$form at $(\mathbb{C}^3,0)$ invariant by a vector field  in $X\in \Pf$. Then $\omega$ is complex hyperbolic.
\end{prop}
\begin{proof}
We apply  Proposition \ref{proposicaoch}.  If the dimensional type is two, it is straight to see that   $\omega$ is simple complex hyperbolic. If the dimensional type is three, then, except for a
possible resonance of its  residues, $\omega$
has the form of a simple complex hyperbolic singularity.
However,  these resonances  can be eliminated by punctual or monoidal blow-ups \cite{cano2004,fernandez-duque2015}, obtaining simple singularities of   complex hyperbolic type. We then conclude that $\omega$ is complex hyperbolic.
\end{proof}

Next result exemplifies how  vector fields and codimension one foliations satisfying
  a relation of tangency   can be geometrically entwined.
 Before stating it, we set a definition:
  a germ  holomorphic vector field --- or its associated one-dimensional holomorphic
   foliation --- at  $(\mathbb{C}^3,0)$ has an
 \emph{absolutely isolated} singularity at $0 \in \C^{3}$ if it
 admits a reduction of singularities having only  punctual blow-ups.
 We call the corresponding composition of blow-up maps
 an \emph{absolutely isolated} reduction of singularities.

\begin{teo}
\label{teo-sp-ch}
Let $\F$ be a germ of
one-dimensional holomorphic foliation  at  $(\mathbb{C}^3,0)$ admitting
a non-dicritical absolutely isolated reduction of singularities
whose associated final models
are all strongly diagonalizable.
If $\mathcal{G}$ is a germ of foliation of codimension   one   invariant by
$\mathcal{F}$, then $\mathcal{G}$ is of complex hyperbolic type.
\end{teo}

\begin{proof}
The proof goes by induction on  $n$, the minimal length of all
 absolutely isolated reductions of singularities for  $\mathcal{F} $ as in the theorem's assertion.
If $n=0$ the result follows from Proposition \ref{propprausarnolema}.
Suppose then that  $n>0$  and that the   result is true for one-dimensional foliations having non-dicritical absolutely isolated
reductions of singularities with strongly diangonal final models of length less than $n$. Denote by $\pi:(M,E)\to (\mathbb{C}^3,0)$
the first punctual blow-up of the corresponding   reduction of singularities of $\mathcal{F}$.
If $\mathcal{G}_1 = \pi^{\ast}\mathcal{G}$ were non-singular over the divisor $E = \pi^{-1}(0) \simeq \mathbb{P}^2$, then $\sing(\mathcal{G})$ would be an isolated singularity
at $0 \in \C^{3}$. As a consequence of Malgrange's Theorem \cite{malgrange1976}, in this case $\cl{G}$ would     have a holomorphic first integral,   being of complex hyperbolic type.
We can then suppose that $\sing(\mathcal{G}_{1}) \cap E \neq \emptyset$ and pick
  $p \in \sing(\mathcal{G}_1) \cap E $. If $p\in \sing(\mathcal{F}_1)$, then, by the induction hypothesis, we must have that $\mathcal{G}_1$ is of
complex hyperbolic type at $p$.
Suppose then that   $p$ is regular for $\mathcal{F}_1$.  In this case, since $\mathcal{F}_1$ is tangent to $\mathcal{G}_1$, the foliation $\mathcal{G}_1$ has dimensional type two at $p$ and the leaf of $\mathcal{F}_1$ at $p$    is a curve contained in the  one-dimensional a analytic set $\sing(\mathcal{G}_1)$. Since $E$ is invariant by $\mathcal{F}_1$, the component of $\sing(\mathcal{G}_1)$ containing this leaf  is contained in $E$ and, hence, it is an algebraic curve in $E \simeq \mathbb{P}^2$, that we denote by $\gamma$.  Now, the sum of   Camacho-Sad indices of $\mathcal{F}_{1{\mid E}}$ along $\gamma$ is the self-intersection number $\gamma \cdot \gamma >0$   (see \cite{camacho1982,suwa1995}). This assures the existence of a singularity $q \in \sing(\mathcal{F}_{1{\mid E}})$,
which is obviously a singularity of  $ \mathcal{F}_1$. By the induction hypothesis, $q$ is of  complex  hyperbolic type for $\mathcal{G}_1$. In view of this, the transversal model of $ \mathcal{G}_1$ along (the generic point of) $ \gamma $ is of   complex hyperbolic type, leading to the conclusion that $\mathcal{G}_1$ is of  complex  hyperbolic  type at $p$. We have found that each singularity of $ \mathcal{G}_1$ over $E$ is of  complex hyperbolic type, admitting a complex hyperbolic reduction of singularities.
This means that $\G$ itself has a complex hyperbolic reduction of singularities, being a foliation of  complex hyperbolic type.
\end{proof}


\section{Integrable pencils of $1-$forms}
\label{section-pencil}

The goal of this section is characterize the situation in  which a germ of holomorphic vector field at $(\C^{3},0)$
is tangent to three independent holomorphic foliations. We show that, when this happens, the vector field
is tangent to infinitely many foliations and   it leaves invariant a germ of analytic surface.
To this end, we work with the notion of pencil of integrable 1-forms or pencil of foliations.
We will formulate our results in the broader context of holomorphic foliations of codimension one at $(\C^{n},0)$, $n \geq 3$, that leave invariant   foliations of codimension two.

We start with a definition. Let  $\omega_1$ and $\omega_2$  be independent
 germs of   holomorphic    $1-$forms at $(\mathbb{C}^n,0)$, that
 is, such that $\omega_1 \wedge \omega_2 \neq 0$. The \emph{pencil of $1-$forms}
with \emph{generators} $\omega_1$ and $\omega_2$ is the linear subspace $\cl{P} = \cl{P}(\omega_{1},\omega_{2})$
of the complex vector space of   germs of holomorphic $1-$forms $(\mathbb{C}^n,0)$ formed by all
 $1-$forms $\omega_{(a,b)} = a\omega_1+b\omega_2$, where
$a,b \in \C $.
We have the following lemma, which is a local version of \cite[Lem. 2]{mol2011}:

\begin{lem}
\label{lem-comp-codim2}
If the independent  germs of  holomorphic  $1-$forms $\omega_1$ and $\omega_2$  do not have  common components   of codimension one in their singular sets, then the $1-$forms  $a\omega_1+b\omega_2 \in \cl{P} = \cl{P}(\omega_{1},\omega_{2})$ have singular sets of codimension at least two,
  except possibly for a finite number of values $(a:b) \in \mathbb{P}^{1}$.
\end{lem}
\begin{proof}
Suppose, by contradiction, that the  result is false.
Then, for infinitely many values of $t \in \C$, the $1-$form $\omega_1+t\omega_2 \in \cl{P}$ has some   component  of codimension one in its singular set,
defined by and irreducible $g_t \in \cl{O}_{n}$. Let us consider these values of $t$.
Writing
$\omega_1=\sum_{i=1}^{n}A_idx_i$ and $ \omega_2=\sum_{i=1}^{n}B_idx_i$,
where $A_i, B_i\in \mathcal{O}_n$, we have that, for each pair $i,j$, with $1\leq i < j\leq n$, both  $A_i+tB_i=0$ and $A_j+tB_j=0$ are zero over $g_t=0$, implying  that $A_iB_j-A_jB_i=0$ over this same set. 
However, the fact that $\sing(\omega_1)$ and  $\sing(\omega_2)$ do not have a common component of codimension one imply that
independent functions $g_{t}$ are associated to different values of $t$.
Hence $A_iB_j-A_jB_i \equiv 0$ and, consequently, $A_i/B_i \equiv A_j/B_j$ for each pair $i,j$. By setting $\Phi = A_i/B_i$ --- which is independent of the chosen $i$ --- we have a germ of meromorphic function at $(\C^{n},0)$ such that $\omega_1=\Phi\omega_2$. This    contradicts the fact that $\omega_1$ and $\omega_2$ are independent $1-$forms.
\end{proof}

Consider a pair of germs of holomorphic $1-$forms   $\omega_1$ and $\omega_2$  as in the   lemma.
We say that $\cl{P} = \cl{P}(\omega_{1},\omega_{2})$
is   a \emph{pencil of integrable $1-$forms} or
an \emph{integrable pencil}  if all its elements  are integrable $1-$forms, that
is, $\omega \wedge d \omega = 0$ for all $\omega \in \cl{P}$.
This is equivalent to the following fact,
which will be referred to as \emph{pencil condition}:
\begin{equation}
\label{pencil}
\omega_1\wedge d\omega_2+\omega_2\wedge d\omega_1=0.
\end{equation}
Observe that, after possibly cancelling codimension one components in the singular set,
we associate to each $\omega_{(a,b)} \in \cl{P}$  a germ of singular holomorphic foliation $\cl{F}_{t}$, where
$t = (a:b)\in \pe^{1}$. For this reason, we also treat this object as \emph{pencil of
holomorphic foliations}.
The $2-$form $\omega_1\wedge\omega_2$ is also integrable, defining, after cancelling singular components of codimension one, a singular holomorphic foliation of codimension two which is
tangent to all foliations (associated to $1-$forms) in $\cl{P} = \cl{P}(\omega_{1},\omega_{2})$.
This codimension two foliation  is called \emph{axis} of $\cl{P}$.

\begin{example}
 {\rm (Logarithmic $1-$forms) Take independent irreducible germs of functions $f_{1},\ldots,f_{k} \in \cl{O}_{n}$, where $k \geq 2$. Consider also $(\lambda_{1},\ldots,\lambda_{k})$ and $(\mu_{1},\ldots,\mu_{k})$ two $\C$-linearly independent $k$-uples of numbers in $\C^{*}$.
Then, the holomorphic $1-$forms
\[ \omega_{1} = f_{1} \cdots f_{k}\left( \lambda_{1} \frac{d f_{1}}{f_{1}} \cdots \lambda_{k} \frac{d f_{k}}{f_{k}} \right) \qquad \text{and} \qquad
 \omega_{2} = f_{1} \cdots f_{k}\left( \mu_{1} \frac{d f_{1}}{f_{1}} \cdots \mu_{k} \frac{d f_{k}}{f_{k}} \right)
  \]
are generators of a pencil of integrable $1-$forms. The axis foliation is defined by the $2-$form
\[ \frac{1}{f_{1} \cdots f_{k}} \omega_{1} \wedge \omega_{2} = \sum_{1 \leq i<j \leq k} (\lambda_{i}\mu_{j} - \lambda_{j}\mu_{i}) h_{ij}  df_{i} \wedge df_{j},\]
where $h_{ij} =  f_{1}\cdots \widehat{f_{i}} \cdots \widehat{f_{j}} \cdots f_{k}$
is the product of all gems $f_{\ell}$ with the exception of $f_{i}$ and  $f_{j}$. Remark that
 the germs of analytic hypersurfaces $\{f_{i}=0\}$ are invariant by all foliations in the integrable
 pencil, as well as by the axis foliation.
}\end{example}

Recall that the integrability of a  holomorphic  $1-$form $\omega $ at $(\mathbb{C}^n,0)$ is equivalent to
the following fact: there exists a meromorphic $1-$form  $\theta$ such that
$d\omega=\theta\wedge\omega$. The sufficiency of this condition is clear. To prove its necessity, it is enough to take  a  meromorphic vector field $Y$   such that $i_{Y}\omega=1$,    contract  by $Y$ both sides of the integrability condition $\omega \wedge d \omega = 0$ and take $\theta=-i_{Y}d\omega$.
Now, if $\cl{P} = \cl{P}(\omega_{1},\omega_{2})$ is an integrable pencil, we get
 meromorphic $1-$forms
  $\theta_1$ and $\theta_2$ satisfying
$d \omega_{1} =\theta_1\wedge\omega_1$ and $d \omega_{2} =\theta_2\wedge\omega_2$.
This, inserted   in the pencil condition \eqref{pencil}, becomes
$(\theta_1-\theta_2)\wedge \omega_1\wedge \omega_2 = 0$.
Then, we find germs of meromorphic functions $g_1,g_2$
 at $(\mathbb{C}^n,0)$ such that
$\theta_1-\theta_2=g_1\omega_1-g_2\omega_2$ (see   Proposition \ref{novaprop} below).
If we define $\theta=\theta_1-g_1\omega_1=\theta_2-g_2\omega_2$, it is straight to
see that
\begin{equation}
\label{eq-pencil-curvature}
 d \omega = \theta \wedge \omega \ \ \forall\   \omega \in \cl{P}.
 \end{equation}
The meromorphic $1-$form $\theta$  is uniquely defined by equation \eqref{eq-pencil-curvature}.
Its exterior derivative $d \theta$ is called \emph{ pencil curvature}, denoted by $k(\mathcal{P})$.

Before proceeding, we present the following result:
\begin{prop}
\label{novaprop}
Let $\omega_1$, $\omega_2$ and $\omega_3$ be   independent germs of holomorphic $1-$forms at $(\C^n,0)$, $n \geq 3$.
Suppose that there exists a non-zero holomorphic $2-$form $\eta$,
 locally decomposable
outside its singular set, that is tangent to each of these three $1-$forms,
i.e.,  $\eta \wedge \omega_{i} = 0$ for $i=1,2,3$.
Then there are germs of meromorphic functions   $\lambda_1$ and $\lambda_2$ at
$(\C^{n},0)$ such that  $\omega_3=\lambda_1\omega_1+\lambda_2\omega_2$.
\end{prop}
\begin{proof}
Denote by
 $T_{ij}=\tang(\omega_i,\omega_j)=\omega_i\wedge\omega_j=0$
 the set of tangencies between $\omega_i$ and $\omega_j$. Note that $T_{ij}$ contains $\sing(\omega_{i}) \cup \sing(\omega_{j})$.  Consider the analytic  set $S = \sing(\eta) \cup T_{12} \cup T_{13} \cup T_{23}$. In a small
 neighborhood of $0 \in \C^{n}$, for each $p \not\in S$,   $\omega_1(p), \omega_2(p)$ and $\omega_3(p)$ define hyperplanes which are pairwise transversal and contain  the subspace of codimension two defined by $\eta(p)$ ($\eta$ is locally decomposable).  Then, by elementary linear algebra, for each $p$ outside $S$, we can write
 $\omega_3  =\lambda_1  \omega_1  +\lambda_2  \omega_2$ ,
 for some uniquely defined $\lambda_1 , \lambda_2  \in \C$. We thus have functions
$\lambda_1$ and $\lambda_2$ defined outside $S$. Wedging the above expression by $\omega_{2}$,
we find $\omega_3 \wedge \omega_{2}  =\lambda_1  \omega_1 \wedge \omega_{2}$. Hence,
$\lambda_{1}$ can also be obtained as
a quotient between a  coefficient  of $\omega_3\wedge\omega_2$ and the corresponding coefficient of $\omega_1\wedge\omega_2$. This shows that it has a meromorphic extension to a neighborhood of $0 \in \C^{n}$, still
denoted by $\lambda_{1}$.
The same reasoning applies to $\lambda_{2}$. By analytic continuation, the relation $\omega_3=\lambda_1\omega_1+\lambda_2\omega_2$
  holds in a neighborhood of $0 \in \C^{n}$, proving the proposition.
\end{proof}

Next, in the framework of  the previous  result,
  we add  integrability  as an ingredient.
We obtain that if a distribution of codimension two is tangent to three independent foliations, then it
is tangent to infinitely many foliations that are in a pencil. More precisely, we have:

\begin{prop}
\label{proposicaodasetas}
Let $\omega_1$, $\omega_2$ and $\omega_3$ be  independent germs of integrable  $1-$forms
at $(\C^{n},0)$, $n \geq 3$, with singular sets of   codimension at least two.
Suppose that there exists a non-zero holomorphic $2-$form $\eta$,
 locally decomposable
outside its singular set, that is tangent to each $\omega_{i}$,
for $i=1,2,3$.
 Then  $\omega_1$, $\omega_2$  and $\omega_3$  define foliations that are in a pencil.
 Furthermore, $\eta$ is integrable, defining   the axis foliation of this pencil.
\end{prop}

\begin{proof}
We start by applying Proposition \ref{novaprop}, finding that $\omega_3=\lambda_1\omega_1+\lambda_2\omega_2$, where $\lambda_1$ and $\lambda_2$ are germs of   meromorphic functions in $(\mathbb{C}^n,0)$. Write $\lambda_i=\psi_i/\varphi_i$, $i=1,2$, with $\psi_i,\varphi_i\in \mathcal{O}_n$ without common factors. Let $\varphi={\rm lcm}(\varphi_1,\varphi_2)$,
where ${\rm lcm}$ denotes the least common multiple.  We then have
\[
\varphi\omega_3=\varphi\lambda_1\omega_1+\varphi\lambda_2\omega_2.
\]
Writing
\[
\eta_1=\varphi\lambda_1\omega_1 = \frac{\varphi}{\varphi_{1}} \psi_{1} \omega_1, \qquad  \eta_2=\varphi\lambda_2\omega_2 = \frac{\varphi}{\varphi_{2}} \psi_{2} \omega_2 \qquad \text{and}  \qquad \eta_3=\varphi\omega_3 ,
\]
we have three integrable $1-$forms, defining the same foliations as $\omega_1$, $\omega_2$ and $\omega_3$,
satisfying
\begin{equation}
\eta_3=\eta_1+\eta_2,
\label{novaeta1}
\end{equation}
so that the pencil condition \eqref{pencil} holds for $\eta_1$ and  $\eta_2$.
Thus, $\cl{P}(\eta_{1},\eta_{2})$
will be an integrable pencil of $1-$forms if its
 generic element has a singular set of codimension at least two.
This   will follow  straight from Lemma \ref{lem-comp-codim2}
if we prove that
$\sing(\eta_1)$ and $\sing(\eta_2)$  do not have common components of codimension one.
Indeed, looking at    \eqref{novaeta1}, the  possible common irreducible components of codimension one of $\sing(\eta_1)$ and $\sing(\eta_2)$ are also components of  $\sing(\eta_3)$.
Fix an equation for such a component. It must be a factor of    $\varphi$, since
$\codim\, \sing(\omega_3) \geq 2$.
By definition of least common multiple, it cannot be a factor of both $\varphi/\varphi_{1}$ and $\varphi/\varphi_{2}$.
Suppose, for instance, that it is not a factor of $\varphi/\varphi_{1}$. Then it is evidently a factor of $\varphi_{1}$ and,
since it defines a component of zeroes of $\eta_{1}$, it   must be also a factor of $\psi_{1}$.
This gives a contradiction, since $\psi_{1}$ and  $\varphi_{1}$ where chosen without common factors.
Finally, the distributions of codimension two subspaces induced by $\eta$ and by the
integrable $2-$form $\eta_{1} \wedge \eta_{2}$ coincide outside
$\sing(\eta) \cup \tang(\eta_{1},\eta_{2})$, giving the last part of the proposition's statement.
\end{proof}

In the sequel we present  a characterization of pencils of integrable $1-$forms at $(\C^{n},0)$.
It is a local version of a result by D. Cerveau   on pencils of foliations in $\pe^{3}$   \cite{cerveau2002}. Our proof essentially follows the same arguments, adapting them to the
local setting.
\begin{teo}
\label{pencil-foliations}
Let $\cl{P}$ be  pencil of integrable $1-$forms  at $(\C^{n},0)$, $n \geq 3$.   Then, at
least one of the following conditions is satisfied:
\begin{enumerate}[label=(\alph*)]
\item  There exists a closed meromorphic
$1-$form $\theta$ such that $d \omega = \theta \wedge \omega$ for every $1-$form
$\omega \in \cl{P}$. When $\theta$ is holomorphic, all foliations in $\cl{P}$ admit  holomorphic  first integrals.
\item   The axis foliation of $\cl{P}$ is tangent to the levels of a non-constant meromorphic function.
    \end{enumerate}
In particular, there exists a germ of hypersurface at $(\C^{n},0)$ that is tangent to the
axis foliation  of $\cl{P}$.
\end{teo}

\begin{proof}
The   two cases in the assertion correspond  to the pencil curvature $k(\mathcal{P})$ being zero or non-zero.

\smallskip \par \noindent  {\bf \underline{Case 1}} :  $k(\mathcal{P})=0$, that is,  the $1-$form $\theta$ in \eqref{eq-pencil-curvature} is closed.
In the purely
  meromorphic case, we can write, from \cite{cerveau1982},
\[
\theta=\sum_{i=1}^k \lambda_i\dfrac{df_i}{f_i}+d\left(\dfrac{h}{f_1^{n_1}\cdots f_k^{n_k}}\right),
\]
where  $ f_i \in \cl{O}_{n}$ are irreducible equations of the components of the polar set of $\theta$, $h \in \cl{O}_{n}$,   $\lambda_i\in \mathbb{C}$, $n_i\in \mathbb{N}$ for $i=1,\cdots,k$, with $\lambda_{i} = 0$ only if
$n_{i} > 0$. Condition
\eqref{eq-pencil-curvature} then says that each hypersurface $f_{i} = 0$ in invariant by
every $\omega \in \cl{P}$ and, hence, also by the axis of $\cl{P}$.

  Suppose, on the other hand, that $\theta$  holomorphic.
Since $d\theta=0$, there exists $h\in \mathcal{O}_n$ such that $\theta=dh$ and hence
$d\omega= dh\wedge\omega$ for every $\omega \in \cl{P}$.
Setting  $h^{\prime}=exp(h) \in \cl{O}_{n}^{*}$, we have
$dh^{\prime}/h^{\prime}= dh$ and thus
 $d\omega = \left(dh^{\prime}/ h^{\prime} \right) \wedge \omega$.
Then
$ d\left(\omega/ h^{\prime} \right) = 0 $, implying that
 there exists $f \in \cl{O}_{n}$ such that  $\omega / h =df$. Hence, each foliation in $\cl{P}$ has
 a holomorphic first integral, and this implies that the axis of $\cl{P}$ is completely integrable, that is, it has two independent  holomorphic
 first integrals. In particular, the axis foliation   leaves  invariant germs of analytic hypersurfaces.

\smallskip \par \noindent  {\bf {\underline{Case 2}}}:  $k(\mathcal{P})\neq 0$.
Let $\omega_{1}$ and $\omega_{2}$ be generators of the pencil.
We have the following:

\begin{assertion}
There exists a germ of meromorphic function $\alpha$ at $(\mathbb{C}^n,0)$ such that
\begin{equation}
d\theta=\alpha \omega_1\wedge \omega_2.
\label{alfa}
\end{equation}
\end{assertion}
\par \noindent \emph{Proof of the Assertion.}\
First note that, taking differentials in both sides of
 $d\omega_i=\theta\wedge \omega_i$,  we have
 $d\theta\wedge \omega_i = 0$   for $i=1,2$.
Now, let
 $q$ be a point near $0\in \mathbb{C}^n$ such that $(\omega_1\wedge \omega_2)(q)\neq 0$.
This means that  $\omega_1$ and $\omega_2$ are non-singular and linearly independent at $q$, so that
we can find analytic coordinates
 $(x_1,x_2,\cdots, x_n)$   at $q$ such that $ \omega_1=A_1dx_1$ and $\omega_2=A_2dx_2$, where  $A_1,A_2$
 are invertible germs of holomorphic functions at $q$.
Write $d\theta=B_{1,2}dx_1\wedge dx_2+ \sum_{(i,j) \neq (1,2)} B_{i,j}dx_i\wedge dx_j$, where each $B_{i,j}$ is a germ of the meromorphic function at $q$.  Since $ d\theta\wedge \omega_1=0$ and $ d\theta\wedge \omega_2=0$, we must have $B_{i,j}=0$ whenever $(i,j) \neq (1,2)$. Then, at $q$,
\[
d\theta=B_{1,2}dx_1\wedge dx_2=\frac{B_{1,2}}{A_1A_2} ( A_1  dx_1)\wedge ( A_2 dx_2)= \alpha  \omega_1\wedge \omega_2,
\]
where $\alpha$ is a germ of the meromorphic function at $q$. In this way, we produce  a meromorphic function $\alpha$, defined outside the set of tangencies   $\tang(\omega_{1},\omega_{2})$, which, by comparing
coefficients of $d\theta$ and $\omega_1\wedge \omega_2$ as we did in the proof of Proposition \ref{novaprop}, can be
extended to a meromorphic function defined in a neighborhood of $0 \in \C^{n}$. This proves the assertion.
\qed

\medskip
Now, we split our analysis in two subcases:

\smallskip \par \noindent  {\bf {\underline{Subcase 2.1}}}:
$\alpha$  is constant.  We claim that the axis  foliation of $\cl{P}$ is tangent to the levels of a meromorphic (possibly holomorphic) function.
Indeed,
  $\alpha \neq 0$, since $k(\mathcal{P})\neq 0$, so that, from the   exterior derivative of \eqref{alfa},
we get  that
$d\omega_1\wedge\omega_2-\omega_1\wedge d\omega_2  = 0$.
This, together with the pencil condition \eqref{pencil}, gives
\begin{equation*}
\omega_2\wedge d\omega_1=\omega_1\wedge d\omega_2=0.
\end{equation*}
Hence, using \eqref{eq-pencil-curvature}, we find that  $\theta \wedge \omega_1 \wedge \omega_2=0$.
Applying Proposition \ref{novaprop} to the $1-$forms $\omega_{1}, \omega_{2}$ and $\omega_{3} = \theta$
 and to the $2-$form $\eta = \omega_1 \wedge \omega_2$, we find    germs of meromorphic functions $\mu_1$ and $\mu_2$ at  $(\mathbb{C}^n,0)$ satisfying
\begin{equation*}
\theta=\mu_1\omega_1+\mu_2\omega_2.
\end{equation*}
Inserting this in \eqref{eq-pencil-curvature}, we get
\begin{equation*}
d\omega_1 =  -\mu_2\omega_1\wedge\omega_2
\qquad \text{and} \qquad
d\omega_2=   \mu_1\omega_1\wedge\omega_2.
\end{equation*}
If $\mu_1=0$ then $d\omega_2=0$, and then $\omega_2=d(g)$  for some $g\in\mathcal{O}_n$, which turns out
 to be a  holomorphic first  integral for the axis of $\cl{P}$. On the other hand, when   $\mu_{1} \neq 0$,
 the above equations give
$d\omega_1= -(\mu_2 /\mu_1)d\omega_2$,
which, by differentiation, yields
\begin{equation*}
  d(\mu_2/\mu_1)\wedge\omega_1\wedge\omega_2 = 0.
\end{equation*}
When $\mu_2/\mu_1$ is non-constant, we have at once that  $\mu_2/\mu_1$  is a meromorphic  first integral for the axis of $\cl{P}$. When $\mu_2/\mu_1 = c$ for some $c \in \C$, we have
$d\omega_1=-cd\omega_2$. That is to say,
 $\omega_1 + c \omega_2$ is a $1-$form in \cl{P} which is closed, and hence exact, yielding once again a   holomorphic first integral for the axis foliation.

\smallskip \par \noindent  {\bf {\underline{Subcase 2.2}}}:
  $\alpha$  is non-constant. Taking the exterior derivative of \eqref{alfa} and using
 \eqref{eq-pencil-curvature}, we obtain
\begin{equation}
\left(\frac{d\alpha}{2\alpha} + \theta \right) \wedge \omega_1 \wedge \omega_2 = 0.
\label{dteta3}
\end{equation}
Now, applying Proposition \ref{novaprop} to $\omega_{1}, \omega_{2},\omega_{3} = d \alpha/2 \alpha + \theta$ and $\eta = \omega_{1} \wedge \omega_{2}$,  we  can find $k_1$ and $k_2$, germs of meromorphic functions at $(\C^{n},0)$,  such that
\begin{equation}
\dfrac{d\alpha}{2\alpha}+\theta=k_1\omega_1+k_2\omega_2.
\label{eq-dteta2}
\end{equation}
Observe that, since $k(\cl{P}) \neq 0$, we have that either $k_1$ or $k_2$ is non-zero.
Taking the exterior derivative and  applying \eqref{eq-pencil-curvature}, we obtain
\begin{equation*}
d\theta = (k_1\theta+dk_1)\wedge\omega_1+(k_2\theta+dk_2)\wedge\omega_2.
\end{equation*}
This, wedged by    $ \omega_1$ and $\omega_2$, gives, respectively,
\begin{equation}
\left(\theta+\dfrac{dk_2}{k_2}\right) \wedge \omega_1\wedge \omega_2=0
\qquad \text{and} \qquad
\left(\theta+\dfrac{dk_1}{k_1} \right) \wedge \omega_1\wedge\omega_2=0.
\end{equation}
Subtracting   \eqref{dteta3},   we obtain, respectively,
\begin{equation} \label{eq1}
\left(-\dfrac{1}{2}\dfrac{d\alpha}{\alpha}+\dfrac{dk_2}{k_2}\right)\wedge \omega_1\wedge\omega_2=0
\qquad \text{and} \qquad
\left(-\dfrac{1}{2}\dfrac{d\alpha}{\alpha}+\dfrac{dk_1}{k_1}\right)\wedge \omega_1\wedge\omega_2=0.
\end{equation}
This allows us to conclude that
he meromorphic functions  $k_1^2/\alpha$ and $k_2^2/\alpha$
are constant on the
leaves of the axis foliation of $\cl{P}$. If one of these two functions is non-constant, we have
a meromorphic first integral for the axis foliation and the proof of the theorem
is accomplished. We claim that this actually happens.
Indeed,
suppose, by contradiction, that both $k_1^2/\alpha$  and  $k_2^2/\alpha$ are constant.
This   would imply
that $k_1/k_2$ (if $k_{2} \neq 0$) is also constant.  Writing $k_1/k_2=c_1 \in \C$,
 we get, from  \eqref{eq-dteta2},
\begin{equation}
\dfrac{d\alpha}{2\alpha}+\theta=k_2(c_1\omega_1+\omega_2)
\end{equation}
and, since $k^2_2/\alpha$ is constant, we find that
\begin{equation}
\label{eq-theta-k2}
\theta= -\dfrac{dk_2}{k_2}+ k_2(c_1\omega_1+\omega_2).
\end{equation}
Then,
applying \eqref{eq-pencil-curvature} to $ c_1\omega_1+\omega_2 \in \cl{P}$ ,
we find
\begin{equation}
d(c_1\omega_1+\omega_2)=-\dfrac{dk_2}{k_2}\wedge (c_1\omega_1+\omega_2).
\end{equation}
This gives that $k_2(c_1\omega_1+\omega_2)$ is closed. By \eqref{eq-theta-k2},
we would have $k(\cl{P}) = d \theta = 0$, reaching a contradiction.
 \end{proof}

Let us put  the previous discussion our initial three-dimensional context:

\begin{teo}
\label{teo-tangent-to-pencil}
Let $X$ be a  germ  of holomorphic vector field at $(\mathbb{C}^3,0)$  tangent to three independent foliations of codimension one. Then there exists  an integrable pencil $\mathcal{P}$ such that
  $X$ is tangent to all foliations in $\cl{P}$.   Furthermore, at least one of the following
  two conditions holds:
\begin{enumerate}
\item[i)] there exists a closed meromorphic $1-$form $\theta$ such that $d \omega = \theta \wedge \omega$ for every   $\omega \in \mathcal{P}$;
\item[ii)]   $X$ has    a non-constant meromorphic (possibly holomorphic) first integral.
\end{enumerate}
In particular, there exists a germ of analytic surface at $0 \in \C^{3}$ which is invariant by $X$.
\end{teo}

We close this article with an example.

\begin{example}
{\rm (Jouanolou's Example) \
Consider, for $m \geq 2$, the vector field $X = x_{3}^{m} \partial/\partial x_{1} +
x_{1}^{m} \partial/\partial x_{2}  +
x_{2}^{m} \partial/\partial x_{3}$.
Then $\omega = i_{R}i_{X} \Omega$, where $\Omega = dx_{1}\wedge dx_{2} \wedge dx_{3}$
is the volume form and
$R = x_{1} \partial/\partial x_{1} +
x_{2} \partial/\partial x_{2}  +
x_{3} \partial/\partial x_{3}$ is the radial vector field,
is an integrable $1-$form, with homogeneous coefficients, that is invariant by $X$.
Since $i_{R}\omega = 0$,   $\omega$ defines a foliation on the complex projective plane
$\mathbb{P}^{2}_{\C}$ which leaves invariant no algebraic curve \cite{jouanolou1979}.
This is equivalent to saying that, at $0 \in \C^{3}$, the vector vector field $X$ leaves invariant no homogeneous
surface --- that is, a surface defined by the vanishing of a homogeneous polynomial in the
coordinates $(x_{1},x_{2},x_{3})$. Since $X$ itself is a homogeneous vector field, this implies
that $X$ is not tangent to any germ of analytic surface at $0 \in \C^{3}$.
By Theorem \ref{teo-tangent-to-pencil}, $X$ is cannot be tangent to three independent foliations.
}\end{example}


\bibliographystyle{plain}
\bibliography{referencias}

\medskip \medskip
\medskip \medskip
\noindent
Dan\'ubia Junca   \\
Departamento de Matem\'atica \\
Universidade Federal de Itajub\'a \\
Rua Irmã Ivone Drumond, 200\\
35903-087
Itabira - MG \\
BRAZIL \\
danubia@unifei.edu.br

\medskip \medskip
\medskip \medskip
\medskip \medskip

\noindent
Rog\'erio   Mol \\
Departamento de Matem\'atica \\
Universidade Federal de Minas Gerais \\
Av. Ant\^onio Carlos, 6627
\ C.P. 702 \\
30123-970 -
Belo Horizonte - MG \\
BRAZIL \\
rmol@ufmg.br

\end{document}